\documentclass[12pt,reqno]{amsart}
\usepackage{amsmath,amsfonts,graphicx,amscd,amssymb,epsf,amsthm,enumerate,parskip,mathrsfs,color,ltablex,blindtext}  
 \usepackage[a4paper, margin=1in]{geometry}
\usepackage{times}
 
\theoremstyle{definition}
\newtheorem{theorem}{Theorem}[section]

\newtheorem{lemma}[theorem]{Lemma}

 \numberwithin{equation}{section}
\numberwithin{equation}{section}
 
\newcommand{\di}{\displaystyle}
\newcommand{\pa}{\partial}

\newcommand{\mb}{\mathbb}
\newcommand{\mf}{\mathbf}
\newcommand{\rvline}{\hspace*{-\arraycolsep}\vline\hspace*{-\arraycolsep}}
\usepackage{cite}
\usepackage[colorlinks,citecolor=blue]{hyperref}

\begin{document}
  
\title{The Brioschi Formula for the Gaussian Curvature}
 \author{Lee-Peng Teo}
\address{Department of Mathematics, Xiamen University Malaysia\\Jalan Sunsuria, Bandar Sunsuria, 43900, Sepang, Selangor, Malaysia.}
\email{lpteo@xmu.edu.my}
\begin{abstract}
The Brioschi formula   expresses the Gaussian curvature $K$ in terms of the functions $E, F$ and $G$ in local coordinates of a surface $S$. This implies the Gauss' theorema egregium, which says that the Gaussian curvature just depends on angles, distances, and their rates of change. 

In most of the textbooks, the Gauss' theorema egregium was proved as a corollary to the derivation of the  Gauss equations,  a set of equations expressing $EK, FK$ and $GK$ in terms of the Christoffel symbols. The Christoffel symbols can be expressed in terms of $E$, $F$ and $G$. In principle, one can derive the Brioschi formula from the Gauss equations after some tedious calculations.

In this note, we give a direct elementary proof of the Brioschi formula without using Christoffel symbols. The key to the proof are properties of matrices and determinants. 
\end{abstract}
\maketitle

 \section{Introduction}
 
 In the classical theory of surfaces in   $\mb{R}^3$, the Gauss' theorema egregium (remarkable theorem) states that the Gaussian curvature of a surface $S$ is invariant under local isometry. In layman terms, the Gaussian curvature is not changed if one bends the surface without stretching it.
 
 This theorem is remarkable since it is not obvious from the definition of the Gaussian curvature. One way to define the Gaussian curvature is to use the first fundamental form and the second fundamental form. Let $Edu^2+2Fdudv+Gdv^2$ and $Ldu^2+2Mdudv+Ndv^2$ be respectively the first fundamental form and the second fundamental form of the surface in local coordinates $(u,v)$. The functions $E, F$ and $G$ measures the angles and  distances, and so they are invariant under local isometries. The functions $L, M$ and $N$ measures the stretching. The   Gaussian curvature $K$ is given by
 \[K=\frac{LN-M^2}{EG-F^2}.\]
 The remarkable theorem of Gauss says that $K$ is expressible purely in terms of the functions $E, F$ and $G$ and their partial derivatives. The usual way to prove this theorem (see for example \cite{DoCarmo, Pressley, Kobayashi, ONeil, Tapp}) is to prove the Gauss equations which express $K$ in terms of $E, F, G$ and the Christoffel sysmbols $\Gamma_{ij}^k$, where $i, j, k\in \{1,2\}$. Since the Christoffel symbols are expressible in terms of $E, F, G$ and their partial derivatives, the remarkable theorem follows.
 
 The Brioschi formula for the Gaussian curvature gives the explicit formula of $K$ in terms of $E, F$, $G$ and their partial derivatives. It can be derived after some tedious calculations from the Gauss equations. 
 In this note, we give an elementary proof of the Brioschi formula which does not need to make use of the Christoffel symbols.

 \section{Classical Theory of Surfaces}
 In this section, we review some classical theory of surfaces. These materials can be found in standard textbooks such as \cite{DoCarmo, Pressley, Kobayashi, ONeil, Tapp}.
 
 Throughout this note, we use the Euclidean metric in $\mb{R}^3$ induced by the inner product
\[\langle \mathbf{a}, \mathbf{b}\rangle =\mathbf{a}\cdot\mathbf{b}=a_1b_1+a_2b_2+a_3b_3\] for two vectors $\mathbf{a}=(a_1, a_2, a_3)$ and $\mathbf{b}=(b_1, b_2, b_3)$ in $\mb{R}^3$.

Let $S$ be a surface in $\mb{R}^3$. 
Given a point $\mathbf{x}_0$ on the surface, there is an open subset $\mathcal{O}$ of $\mb{R}^2$, an open subset $U$ of $\mb{R}^3$ containing $\mathbf{x}_0$,  and a one-to-one smooth map $\boldsymbol{\sigma}:\mathcal{O}\to S\cap U$ which provides local coordinates $u$ and $v$ to the point $\boldsymbol{\sigma}(u,v)$ on $U\cap S$. The map $\boldsymbol{\sigma}:\mathcal{O}\to S$ is called a coordinate patch of $S$. For every $(u,v)$ in $\mathcal{O}$, the vectors $\boldsymbol{\sigma}_u(u,v)$ and $\boldsymbol{\sigma}_v(u,v)$ are linearly independent vectors in $\mb{R}^3$ which form a basis of the tangent space $T_{\boldsymbol{\sigma}(u,v)}S$ at $\boldsymbol{\sigma}(u,v)$. Define
 \begin{equation}\label{20240401_5}E =\langle \boldsymbol{\sigma}_u, \boldsymbol{\sigma}_u\rangle, \quad 
 F =\langle \boldsymbol{\sigma}_u, \boldsymbol{\sigma}_v\rangle, \quad
G =\langle \boldsymbol{\sigma}_v, \boldsymbol{\sigma}_v\rangle.\end{equation}
Then $Edu^2+2Fdudv+Gdv^2$ is called the first fundamental form of the coordinate patch $\boldsymbol{\sigma}:\mathcal{O}\to S$  of the surface. It defines the quadratic form $\mathcal{Q}_I: T_{\boldsymbol{\sigma}(u,v)}S \to \mb{R}$ corresponding to the inner product on $T_{\boldsymbol{\sigma}(u,v)}S $. For any tangent vector $\mathbf{w}$ in  $T_{\boldsymbol{\sigma}(u,v)}S $,
\[\mathcal{Q}_I(\mathbf{w}) =\langle \mathbf{w}, \mathbf{w}\rangle.\]
If  $\alpha$ and $\beta$ are such that 
\[\mathbf{w}=\alpha\boldsymbol{\sigma}_u(u,v)+\beta \boldsymbol{\sigma}_v(u,v),\]
then 
\[\mathcal{Q}_I(\mathbf{w})=\begin{bmatrix}\alpha & \beta\end{bmatrix}\mathcal{F}_I\begin{bmatrix}\alpha\\\beta\end{bmatrix}=E\alpha^2+2F\alpha\beta+G\beta^2,\] where $\mathcal{F}_I$ is the symmetric matrix
\[\mathcal{F}_I=\begin{bmatrix}E & F \\ F& G\end{bmatrix}.\]

At the point $\boldsymbol{\sigma}(u,v)$ on the surface, the vector
\begin{equation}\label{20240401_3}\mathbf{N}=\frac{\boldsymbol{\sigma}_u\times  \boldsymbol{\sigma}_v}{\Vert \boldsymbol{\sigma}_u\times  \boldsymbol{\sigma}_v\Vert}=\frac{\boldsymbol{\sigma}_u\times  \boldsymbol{\sigma}_v}{\sqrt{EG-F^2}}\end{equation}
is   the standard unit normal vector orthogonal to the tangent space  $T_{\boldsymbol{\sigma}(u,v)}S $.
The functions $L$, $M$ and $N$ are defined as
\begin{equation}\label{20240401_2}
L=\langle \boldsymbol{\sigma}_{uu}, \mathbf{N}\rangle, \quad M=\langle \boldsymbol{\sigma}_{uv}, \mathbf{N}\rangle, \quad N=\langle \boldsymbol{\sigma}_{vv}, \mathbf{N}\rangle.
\end{equation}
The quadratic form defined by the second fundamental form $Ldu^2+2Mdudv+Ndv^2$ is $\mathcal{Q}_{I\!I}: T_{\boldsymbol{\sigma}(u,v)}S \to \mb{R}$, 
\[\mathcal{Q}_{I\!I}(\alpha\boldsymbol{\sigma}_u +\beta \boldsymbol{\sigma}_v )=\begin{bmatrix}\alpha & \beta\end{bmatrix}\mathcal{F}_{I\!I}\begin{bmatrix}\alpha\\\beta\end{bmatrix}=L\alpha^2+2M\alpha\beta+N\beta^2,\] where $\mathcal{F}_{I\!I}$ is the symmetric matrix
\[\mathcal{F}_{I\!I}=\begin{bmatrix}L & M \\ M& N\end{bmatrix}.\]
To put this quadratic form in perspective, we need to introduce the Gauss map and the Weingarten map.

 Let $S^2$ be the standard two sphere $x^2+y^2+z^2=1$. It can be considered as the space that contains all the unit vectors in $\mb{R}^3$.
  At any point $\mathbf{w}_0 $ on $S^2$, the standard unit normal vector to $S^2$ at $\mathbf{w}_0$ is the vector $\mathbf{w}_0$ itself. The Gauss map  $\mathcal{G}:S\to S^2$ is a map from an orientable surface $S$ to $S^2$ mapping the point $\mathbf{x}_0$ to the unit normal vector  $\mathbf{w}_0=\mathbf{N}_{\mathbf{x}_0}$ at $\mathbf{x}_0$.

  Note that the tangent space $T_{\mathbf{x}_0}S$ of $S$ at $\mathbf{x}_0$ is naturally identified with the tangent space $T_{\mathbf{w}_0}S^2$ of $S^2$ at $\mathbf{w}_0$ since the standard unit normal vector to $S^2$ at $\mathbf{w}_0$ is $\mathbf{w}_0$ itself. The Weingarten map $\mathcal{W}_{\mathbf{x}_0,S}$ at $\mathbf{x}_0$ is  defined as $-D_{\mathbf{x}_0}\mathcal{G}$, the negative of the derivative of the  Gauss map at $\mathbf{x}_0$. Using the natural  identification of  $T_{\mathbf{x}_0}S$  with $T_{\mathbf{w}_0}S^2$, we can consider $\mathcal{W}_{\mathbf{x}_0, S}$ as a map from $T_{\mathbf{x}_0}S$ to itself. Since the derivative map $D_{\mathbf{x}_0}\mathcal{G}$   is a linear transformation between tangent spaces, the Weingarten map $\mathcal{W}_{\mathbf{x}_0, S}:T_{\mathbf{x}_0}S\to T_{\mathbf{x}_0}S$  is a linear transformation.  By definition,
  \begin{align*}
  \mathcal{W} (\boldsymbol{\sigma}_u)=-\mathbf{N}_u,\quad \mathcal{W}(\boldsymbol{\sigma}_v)=-\mathbf{N}_v.
  \end{align*}  
If
\[\mathbf{w}=\alpha \boldsymbol{\sigma}_u+\beta \boldsymbol{\sigma}_v,\] then
\begin{align*}
\langle \mathcal{W}(\mathbf{w}), \mathbf{w}\rangle 
&=\alpha^2\langle \mathcal{W}(\boldsymbol{\sigma}_u), \boldsymbol{\sigma}_u\rangle+\alpha\beta\left(\langle \mathcal{W}(\boldsymbol{\sigma}_u), \boldsymbol{\sigma}_v\rangle+\langle \mathcal{W}(\boldsymbol{\sigma}_v), \boldsymbol{\sigma}_u\rangle\right)+\beta^2\langle \mathcal{W}(\boldsymbol{\sigma}_v), \boldsymbol{\sigma}_v\rangle\\
&=-\alpha^2\langle  \mathbf{N}_u, \boldsymbol{\sigma}_u\rangle-\alpha\beta\left(\langle  \mathbf{N}_u, \boldsymbol{\sigma}_v\rangle+\langle   \mathbf{N}_v, \boldsymbol{\sigma}_u\rangle\right)-\beta^2\langle    \mathbf{N}_v, \boldsymbol{\sigma}_v\rangle.
\end{align*}
Using the fact that $\boldsymbol{\sigma}_u$ and $\boldsymbol{\sigma}_v$ are orthogonal to $\mathbf{N}$, we have
\begin{align*}
L=-\langle \boldsymbol{\sigma}_{u}, \mathbf{N}_u\rangle, \quad M=-\langle \boldsymbol{\sigma}_{u}, \mathbf{N}_v\rangle=-\langle \boldsymbol{\sigma}_{v}, \mathbf{N}_u\rangle, \quad N=-\langle \boldsymbol{\sigma}_{v}, \mathbf{N}_v\rangle.
\end{align*}
Therefore, 
\[\langle \mathcal{W}(\mathbf{w}), \mathbf{w}\rangle =L\alpha^2+2M\alpha\beta+N\beta^2=\mathcal{Q}_{I\!I}(\mathbf{w}).\]
In other words, the second fundamental form is the 
quadratic form corresponding to the Weingarten map. Since the matrix $\mathcal{F}_{I\!I}$ is symmetric, the Weingarten map is self-adjoint.

 The Gaussian curvature $K$ of the surface $S$ at $\mathbf{x}_0$ is defined as the determinant of the linear transformation $\mathcal{W}_{\mathbf{x}_0, S}: T_{\mathbf{x}_0}S\to T_{\mathbf{x}_0}S$. If
 \[\mathbf{A}=\begin{bmatrix} a & b\\ c & d\end{bmatrix}\] is the matrix of the linear transformation $\mathcal{W}$ with respect to the basis $\{\boldsymbol{\sigma}_u, \boldsymbol{\sigma}_v\}$, then
 \[K=\det\mathbf{A}=ad-bc.\]By definition, 
 \[\mathcal{W}(\boldsymbol{\sigma}_u)=a\boldsymbol{\sigma}_u+c\boldsymbol{\sigma}_v,\hspace{1cm} \mathcal{W}(\boldsymbol{\sigma}_v)=b\boldsymbol{\sigma}_u+d\boldsymbol{\sigma}_v.\]
 Taking inner products with $\boldsymbol{\sigma}_u$ and $\boldsymbol{\sigma}_v$ respectively, we find that
 \[L=aE+cF, \quad M=aF+cG,\quad M=bE+dF,\quad N=bF+dG.\]In other words,
 \begin{align*}
 \begin{bmatrix} L & M\\ M & N\end{bmatrix}=\begin{bmatrix} E & F\\F & G\end{bmatrix}\begin{bmatrix} a & b\\ c & d\end{bmatrix}.
 \end{align*}Equivalently,
 \[\mathbf{A}=\mathcal{F}_I^{-1}\mathcal{F}_{I\!I}.\]
 Therefore,
 \[K=\det\mf{A}=\frac{\det \mathcal{F}_{I\!I}}{\det\mathcal{F}_I}=\frac{LN-M^2}{EG-F^2}.\]

In the following, let us give a sketch of the usual approach to   Gauss' theorema egregium.
 The set of  vectors $\left\{\boldsymbol{\sigma}_u(u,v), \boldsymbol{\sigma}_v(u,v), \mathbf{N}(u,v)\right\}$ forms a basis of $\mathbb{R}^3$, and $\boldsymbol{\sigma}_u$ and $\boldsymbol{\sigma}_v$ are perpendicular to $\mathbf{N}$.  By definitions \eqref{20240401_2} of $L$, $M$ and $N$, we find that there exist functions $\Gamma_{11}^1$,  $\Gamma_{11}^2$, $\Gamma_{12}^1$,  $\Gamma_{12}^2$, $\Gamma_{22}^1$, $\Gamma_{22}^2$ such that
 \begin{align*}
 \boldsymbol{\sigma}_{uu}&=\Gamma_{11}^1\boldsymbol{\sigma}_u+\Gamma_{11}^2\boldsymbol{\sigma}_v+L\mathbf{N},\\
  \boldsymbol{\sigma}_{uv}&=\Gamma_{12}^1\boldsymbol{\sigma}_u+\Gamma_{12}^2\boldsymbol{\sigma}_v+M\mathbf{N},\\
   \boldsymbol{\sigma}_{vv}&=\Gamma_{22}^1\boldsymbol{\sigma}_u+\Gamma_{22}^2\boldsymbol{\sigma}_v+N\mathbf{N}.
 \end{align*}The functions $\Gamma_{11}^1$,  $\Gamma_{11}^2$, $\Gamma_{12}^1$,  $\Gamma_{12}^2$, $\Gamma_{22}^1$, $\Gamma_{22}^2$ are known as the Christoffel symbols. Taking inner products with $\boldsymbol{\sigma}_u$ and $\boldsymbol{\sigma}_v$ respectively, we find that
 \begin{equation}\label{20240331_1}
 \begin{split}
 \begin{bmatrix} E & F\\ F & G\end{bmatrix}\begin{bmatrix}\Gamma_{11}^1 & \Gamma_{12}^1 & \Gamma_{22}^1\\ \Gamma_{11}^2 & \Gamma_{12}^2 & \Gamma_{22}^2\end{bmatrix}&=\begin{bmatrix}\langle \boldsymbol{\sigma}_{uu}, \boldsymbol{\sigma}_u\rangle &\langle \boldsymbol{\sigma}_{uv}, \boldsymbol{\sigma}_u\rangle & \langle \boldsymbol{\sigma}_{vv}, \boldsymbol{\sigma}_u\rangle \\ \langle \boldsymbol{\sigma}_{uu}, \boldsymbol{\sigma}_v\rangle &\langle \boldsymbol{\sigma}_{uv}, \boldsymbol{\sigma}_v\rangle & \langle \boldsymbol{\sigma}_{vv}, \boldsymbol{\sigma}_v\rangle\end{bmatrix} \\
 &=\begin{bmatrix}\frac{1}{2}E_u & \frac{1}{2}E_v & F_v-\frac{1}{2}G_u\\F_u-\frac{1}{2}E_v & \frac{1}{2}G_u & \frac{1}{2}G_v \end{bmatrix}.
\end{split} \end{equation}This shows that the Christoffel symbols can be expressed in terms of the functions $E$, $F$, $G$ and their partial derivatives.
 Using the compactibility conditions
 \[(\boldsymbol{\sigma}_{uu})_v=(\boldsymbol{\sigma}_{uv})_u,\hspace{1cm} (\boldsymbol{\sigma}_{uv})_v=(\boldsymbol{\sigma}_{vv})_u,\]
 one can derive the Gauss equations 
  \begin{equation}\label{20240401_6}
  \begin{split}
 EK&=\frac{\pa \Gamma_{11}^2}{\pa v}-\frac{\pa \Gamma_{12}^2}{\pa u}+\Gamma_{11}^1\Gamma_{12}^2+\Gamma_{11}^2\Gamma_{22}^2-\left(\Gamma_{12}^2\right)^2-\Gamma_{12}^1\Gamma_{11}^2, \\
 FK&= \di\frac{\pa \Gamma_{12}^1}{\pa u}-\frac{\pa \Gamma_{11}^1}{\pa v}+\Gamma_{12}^1\Gamma_{12}^2-\Gamma_{22}^1\Gamma_{11}^2, \\
 FK&=\di\frac{\pa \Gamma_{12}^2}{\pa v}-\frac{\pa \Gamma_{22}^2}{\pa u}+\Gamma_{12}^1\Gamma_{12}^2-\Gamma_{22}^1\Gamma_{11}^2, \\
 GK&=\di\frac{\pa \Gamma_{22}^1}{\pa u}  -\frac{\pa \Gamma_{12}^1}{\pa v}+\Gamma_{11}^1\Gamma_{22}^1+\Gamma_{12}^1\Gamma_{22}^2-\left(\Gamma_{12}^1\right)^2-\Gamma_{22}^1\Gamma_{12}^2. \end{split}
 \end{equation}
 Anyone of these equations is sufficient to conclude the remarkable theorem of Gauss.
 
 In principal, one can use the formulas in \eqref{20240401_6} and \eqref{20240331_1} to derive a formula for the Gaussian curvature purely in terms of the functions $E, F, G$ and their partial derivatives. However, this is a tedious calculation and it gives a sense that the introduction of  the Christoffel symbols are indispensable. 

 \section{The   Brioschi Formula}

The Brioschi formula expresses the Gaussian curvature $K$ purely in terms of the functions $E, F, G$ and their partial derivatives. Although the   second fundamental form of a surface depends on the orientation of a surface, the Gaussian curvature does not. It is well-defined even though the surface is not orientable.

The following theorem gives the Brioschi formula.
\begin{theorem}\label{240401_1}
 Let $S$ be a surface, and let $\boldsymbol{\sigma}:\mathcal{O}\to S$ be a coordinate patch with first fundamental form $Edu^2+2Fdudv+Gdv^2$. The Gaussian curvature of $S$ is given by
 \[ K=\frac{\det \mf{B}_1 -\det\mf{B}_2}{(EG-F^2)^2},\]
where 
\begin{align*}\mf{B}_1=\begin{bmatrix} \di -\frac{1}{2}E_{vv}+F_{uv}-\frac{1}{2}G_{uu} &\di \frac{1}{2}E_u &\di  F_u-\frac{1}{2}E_v\\[2ex]
\di F_v-\frac{1}{2}G_u &E & F\\[2ex]
\di \frac{1}{2}G_v & F & G\end{bmatrix},\quad \mf{B}_2=\begin{bmatrix} 0&\di \frac{1}{2}E_v &\di  \frac{1}{2}G_u\\[2ex]
\di  \frac{1}{2}E_v &E & F\\[2ex]
\di \frac{1}{2}G_u & F & G\end{bmatrix}. \end{align*}
\end{theorem}
 
 This section is devoted to the proof of this formula. Note that the Gauss' theorema egregium  is a consequence of this formula. 
 
  Given a vector  $\mf{a}=(a_1, a_2, a_3)$  in $\mb{R}^3$, we  also denote by $\mf{a}$ the corresponding column vector, namely,
 \[\mf{a}=\begin{bmatrix} a_1\\a_2\\a_3\end{bmatrix}.\]
 In this notation, the Euclidean inner product  of the vectors $\mf{a}=(a_1, a_2, a_3)$ and $\mf{b}=(b_1, b_2, b_3)$ can be expressed as
 \[\langle \mathbf{a}, \mathbf{b}\rangle =\mathbf{a}^T\mathbf{b}.\]
If $\mf{c}_1, \mf{c}_2, \mf{c}_3$ are three vectors in $\mb{R}^3$, we denote by
 \[\begin{bmatrix} \mf{c}_1 &\rvline& \mf{c}_2 &\rvline &\mf{c}_3\end{bmatrix}\] the matrix with column vectors $\mf{c}_1$, $\mf{c}_2$ and $\mf{c}_3$. Then
 \[\det\mf{C}=\det\mf{C}^T=\langle\mf{c}_1, \mf{c}_2\times\mf{c}_3\rangle.\]
Given the matrices
 \[\mf{C}=\begin{bmatrix} \mf{c}_1 &\rvline& \mf{c}_2 &\rvline &\mf{c}_3\end{bmatrix}\quad\text{and}\quad
 \mf{D}=\begin{bmatrix} \mf{d}_1 &\rvline& \mf{d}_2 &\rvline &\mf{d}_3\end{bmatrix}, \] the components of $\mf{C}^T\mf{D}$ can be expressed in terms of inner products of their column vectors. Namely,
 \[\mf{C}^T\mf{D}=\begin{bmatrix} \langle\mf{c}_1, \mf{d}_1\rangle &   \langle\mf{c}_1, \mf{d}_2\rangle & \langle\mf{c}_1, \mf{d}_3\rangle \\\langle\mf{c}_2, \mf{d}_1\rangle &   \langle\mf{c}_2, \mf{d}_2\rangle & \langle\mf{c}_2, \mf{d}_3\rangle \\\langle\mf{c}_3, \mf{d}_1\rangle &   \langle\mf{c}_3, \mf{d}_2\rangle & \langle\mf{c}_3, \mf{d}_3\rangle \end{bmatrix}.\]
  
 The following is an elementary lemma which is key to the proof of the Brioschi's formula.
\begin{lemma}\label{20240401_4}
 The  identity  of determinants
   \begin{align*}
 \det \begin{bmatrix} a_1 & b_1 & c_1\\ d_1 & e & f\\h_1 & f & e\end{bmatrix} - \det \begin{bmatrix} a_2 & b_2 & c_2\\ d_2 & e & f\\h_2 & f & e\end{bmatrix}= \det \begin{bmatrix} a_1-a_2 & b_1 & c_1\\ d_1 & e & f\\h_1 & f & e\end{bmatrix} - \det \begin{bmatrix} 0 & b_2 & c_2\\ d_2 & e & f\\h_2 & f & e\end{bmatrix}
   \end{align*}holds.
\end{lemma}
\begin{proof}We need to show that
 \[\det\mf{H}_1-\det\mf{H}_2= \det\mf{H}_3-\det\mf{H}_4,\]where
\begin{align*}
\mf{H}_1&= \begin{bmatrix} a_1 & b_1 & c_1\\ d_1 & e & f\\h_1 & f & e\end{bmatrix},\quad\mf{H}_2=\begin{bmatrix} a_2 & b_2 & c_2\\ d_2 & e & f\\h_2 & f & e\end{bmatrix},\\
\mf{H}_3&= \begin{bmatrix} a_1-a_2 & b_1 & c_1\\ d_1 & e & f\\h_1 & f & e\end{bmatrix}, \quad \mf{H}_4= \begin{bmatrix} 0 & b_2 & c_2\\ d_2 & e & f\\h_2 & f & e\end{bmatrix}.
 \end{align*} 
Using expansion with respect to the first row, we find that
 \[\det\mf{H}_1-\det\mf{H}_3=a_2 \det\begin{bmatrix} e & f\\f & e\end{bmatrix}=\det\mf{H}_2-\det\mf{H}_4.\]The result follows.
\end{proof}
 
 Now we can prove the Brioschi's formula for Gaussian curvature.
 \begin{proof}
 [Proof of Theorem \ref{240401_1}.]
 Since 
 \[K=\frac{LN-M^2}{EG-F^2},\]
the definitions \eqref{20240401_2} of $L, M,  N$ and the definition \eqref{20240401_3} of $\mf{N}$ give
 \begin{align*}
 K=\frac{\langle \boldsymbol{\sigma}_{uu},\boldsymbol{\sigma}_u\times\boldsymbol{\sigma}_v\rangle \langle \boldsymbol{\sigma}_{vv},\boldsymbol{\sigma}_u\times\boldsymbol{\sigma}_v\rangle -\langle \boldsymbol{\sigma}_{uv},\boldsymbol{\sigma}_u\times\boldsymbol{\sigma}_v\rangle^2}{(EG-F^2)^2}.
 \end{align*}
 Let $\mathcal{L}$, $\mathcal{M}$ and $\mathcal{N}$ be the matrices given by
 \[\mathcal{L}=\begin{bmatrix} \boldsymbol{\sigma}_{uu}&\rvline& \boldsymbol{\sigma}_u &\rvline &\boldsymbol{\sigma}_v\end{bmatrix}, \quad 
 \mathcal{M}=\begin{bmatrix} \boldsymbol{\sigma}_{uv}&\rvline& \boldsymbol{\sigma}_u &\rvline &\boldsymbol{\sigma}_v\end{bmatrix},\quad 
 \mathcal{N}=\begin{bmatrix} \boldsymbol{\sigma}_{vv}&\rvline& \boldsymbol{\sigma}_u &\rvline &\boldsymbol{\sigma}_v\end{bmatrix}.\]
Then
 \[K=\frac{\det \mathcal{L}\det\mathcal{N}-\left(\det\mathcal{M}\right)^2}{(EG-F^2)}=\frac{\det (\mathcal{L}^T\mathcal{N})-\det( \mathcal{M}^T\mathcal{M})}{(EG-F^2)^2}.\]
 Now,
 \begin{align*}\mathcal{L}^T\mathcal{N}&=\begin{bmatrix} \langle \boldsymbol{\sigma}_{uu},\boldsymbol{\sigma}_{vv}\rangle & \langle \boldsymbol{\sigma}_{uu},\boldsymbol{\sigma}_{u}\rangle & \langle \boldsymbol{\sigma}_{uu},\boldsymbol{\sigma}_{v }\rangle\\
  \langle \boldsymbol{\sigma}_{u},\boldsymbol{\sigma}_{vv}\rangle & \langle \boldsymbol{\sigma}_{u},\boldsymbol{\sigma}_{u}\rangle & \langle \boldsymbol{\sigma}_{u},\boldsymbol{\sigma}_{v }\rangle\\
   \langle \boldsymbol{\sigma}_{v},\boldsymbol{\sigma}_{vv}\rangle & \langle \boldsymbol{\sigma}_{v},\boldsymbol{\sigma}_{u}\rangle & \langle \boldsymbol{\sigma}_{v},\boldsymbol{\sigma}_{v }\rangle\end{bmatrix}, \\ \mathcal{M}^T\mathcal{M} &=\begin{bmatrix} \langle \boldsymbol{\sigma}_{uv},\boldsymbol{\sigma}_{uv}\rangle & \langle \boldsymbol{\sigma}_{uv},\boldsymbol{\sigma}_{u}\rangle & \langle \boldsymbol{\sigma}_{uv},\boldsymbol{\sigma}_{v }\rangle\\
  \langle \boldsymbol{\sigma}_{u},\boldsymbol{\sigma}_{uv}\rangle & \langle \boldsymbol{\sigma}_{u},\boldsymbol{\sigma}_{u}\rangle & \langle \boldsymbol{\sigma}_{u},\boldsymbol{\sigma}_{v }\rangle\\
   \langle \boldsymbol{\sigma}_{v},\boldsymbol{\sigma}_{uv}\rangle & \langle \boldsymbol{\sigma}_{v},\boldsymbol{\sigma}_{u}\rangle & \langle \boldsymbol{\sigma}_{v},\boldsymbol{\sigma}_{v }\rangle\end{bmatrix}.\end{align*}
   From
 \eqref{20240401_5}, 
   we have
   \begin{gather*}E_u=2\langle \boldsymbol{\sigma}_{uu},\boldsymbol{\sigma}_{u}\rangle,\quad F_u= \langle \boldsymbol{\sigma}_{uu},\boldsymbol{\sigma}_{v }\rangle+ \langle \boldsymbol{\sigma}_{u},\boldsymbol{\sigma}_{uv }\rangle,\quad G_u=2 \langle \boldsymbol{\sigma}_{uv},\boldsymbol{\sigma}_{v }\rangle,\\
   E_v=2\langle \boldsymbol{\sigma}_{uv},\boldsymbol{\sigma}_{u}\rangle,\quad F_v= \langle \boldsymbol{\sigma}_{uv},\boldsymbol{\sigma}_{v }\rangle+ \langle \boldsymbol{\sigma}_{u},\boldsymbol{\sigma}_{vv }\rangle,\quad G_v=2 \langle \boldsymbol{\sigma}_{vv},\boldsymbol{\sigma}_{v }\rangle.\end{gather*}
   It follows that
    \begin{align*}
 \langle  \boldsymbol{\sigma}_{uu},\boldsymbol{\sigma}_{u}\rangle &=\frac{1}{2}E_u,\hspace{2cm} \langle \boldsymbol{\sigma}_{uu},\boldsymbol{\sigma}_{v }\rangle=F_u-\frac{1}{2}E_v,\\
  \langle  \boldsymbol{\sigma}_{uv},\boldsymbol{\sigma}_{u}\rangle &=\frac{1}{2}E_v,\hspace{2cm} \langle \boldsymbol{\sigma}_{uv},\boldsymbol{\sigma}_{v }\rangle=\frac{1}{2}G_u,\\
   \langle  \boldsymbol{\sigma}_{vv},\boldsymbol{\sigma}_{u}\rangle &=F_v-\frac{1}{2}G_u,\hspace{1cm} \langle \boldsymbol{\sigma}_{vv},\boldsymbol{\sigma}_{v }\rangle= \frac{1}{2}G_v.
   \end{align*}
   Therefore,
   \begin{align*}\mathcal{L}^T\mathcal{N}&=\begin{bmatrix} \langle \boldsymbol{\sigma}_{uu},\boldsymbol{\sigma}_{vv}\rangle &\frac{1}{2}E_u & F_u-\frac{1}{2}E_v\\ 
F_v-\frac{1}{2}G_u &   E & F\\
   \frac{1}{2}G_v & F &G\end{bmatrix}, \end{align*}
   \begin{align*} \mathcal{M}^T\mathcal{M} &=\begin{bmatrix}  \langle \boldsymbol{\sigma}_{uv},\boldsymbol{\sigma}_{uv}\rangle &\frac{1}{2}E_v &  \frac{1}{2}G_u\\ 
\frac{1}{2}E_v&   E & F\\
   \frac{1}{2}G_u& F &G\end{bmatrix}.
   \end{align*}
 By Lemma \ref{20240401_4}, we find that
   \begin{align*}
   \det (\mathcal{L}^T\mathcal{N})-\det( \mathcal{M}^T\mathcal{M})=\det\mf{B}_1-\det\mf{B}_2,
   \end{align*}
   where
   \begin{align*}\mf{B}_1=\begin{bmatrix} \di \langle \boldsymbol{\sigma}_{uu},\boldsymbol{\sigma}_{vv}\rangle - \langle \boldsymbol{\sigma}_{uv},\boldsymbol{\sigma}_{uv}\rangle &\di \frac{1}{2}E_u &\di  F_u-\frac{1}{2}E_v\\[2ex]
\di F_v-\frac{1}{2}G_u &E & F\\[2ex]
\di \frac{1}{2}G_v & F & G\end{bmatrix},\quad \mf{B}_2=\begin{bmatrix} 0&\di \frac{1}{2}E_v &\di  \frac{1}{2}G_u\\[2ex]
\di  \frac{1}{2}E_v &E & F\\[2ex]
\di \frac{1}{2}G_u & F & G\end{bmatrix}. \end{align*}
 Since
   \begin{align*}
   E_{vv}&=2\langle \boldsymbol{\sigma}_{uvv},\boldsymbol{\sigma}_{u}\rangle+2\langle \boldsymbol{\sigma}_{uv},\boldsymbol{\sigma}_{uv}\rangle,\\
   F_{uv}&= \langle \boldsymbol{\sigma}_{uuv},\boldsymbol{\sigma}_{v }\rangle +\langle \boldsymbol{\sigma}_{uu},\boldsymbol{\sigma}_{vv }\rangle+ \langle \boldsymbol{\sigma}_{uv},\boldsymbol{\sigma}_{uv }\rangle+ \langle \boldsymbol{\sigma}_{u},\boldsymbol{\sigma}_{uvv }\rangle,\\
   G_{uu}&=2 \langle \boldsymbol{\sigma}_{uuv},\boldsymbol{\sigma}_{v }\rangle+2 \langle \boldsymbol{\sigma}_{uv},\boldsymbol{\sigma}_{vu }\rangle,
   \end{align*}we find that
   \[ \langle \boldsymbol{\sigma}_{uu},\boldsymbol{\sigma}_{vv}\rangle - \langle \boldsymbol{\sigma}_{uv},\boldsymbol{\sigma}_{uv}\rangle=-\frac{1}{2}E_{vv}+F_{uv}-\frac{1}{2}G_{uu}.\]This completes the proof.

 \end{proof}
 
 From the Brioschi formula, it is easy to derive the expression for  the Gaussian curvature $K$ in the special case where $F=0$. We leave this as an exercise to the readers. 
 
 Finally, we would like to remark that   the proof given here is the same as the one presented in the book \cite{Gray}. We believe that this is a more straightforward approach to Gauss' theorema egregium.

\bibliographystyle{amsalpha}
\bibliography{ref}
\end{document}